\documentclass[reqno, 12pt]{amsart}
\pdfoutput=1
\makeatletter
\let\origsection=\section \def\section{\@ifstar{\origsection*}{\mysection}}
\def\mysection{\@startsection{section}{1}\z@{.7\linespacing\@plus\linespacing}{.5\linespacing}{\normalfont\scshape\centering\S}}
\makeatother

\usepackage{amsmath,amssymb,amsthm}
\usepackage{mathrsfs}
\usepackage{mathabx}\changenotsign
\usepackage{dsfont}
\usepackage{bbm}
\usepackage{scalerel}

\usepackage{xcolor}
\usepackage[backref]{hyperref}
\hypersetup{
	colorlinks,
	linkcolor={red!60!black},
	citecolor={green!60!black},
	urlcolor={blue!60!black}
}

\usepackage[open,openlevel=2,atend]{bookmark}

\usepackage[abbrev,msc-links,backrefs]{amsrefs}
\usepackage{doi}

\renewcommand{\PrintDOI}[1]{\doi{#1}}

\usepackage[T1]{fontenc}
\usepackage{lmodern}
\usepackage[babel]{microtype}
\usepackage[english]{babel}
\usepackage{verbatim}

\usepackage{geometry}
\numberwithin{equation}{section}
\numberwithin{figure}{section}

\usepackage{enumitem}

\let\polishlcross=\l
\def\l{\ifmmode\ell\else\polishlcross\fi}

\let\emptyset=\varnothing
\let\setminus=\smallsetminus

\makeatletter
\def\moverlay{\mathpalette\mov@rlay}
\def\mov@rlay#1#2{\leavevmode\vtop{   \baselineskip\z@skip \lineskiplimit-\maxdimen
		\ialign{\hfil$\m@th#1##$\hfil\cr#2\crcr}}}
\newcommand{\charfusion}[3][\mathord]{
	#1{\ifx#1\mathop\vphantom{#2}\fi
		\mathpalette\mov@rlay{#2\cr#3}
	}
	\ifx#1\mathop\expandafter\displaylimits\fi}
\makeatother

\newcommand{\dcup}{\charfusion[\mathbin]{\cup}{\cdot}}

\DeclareFontFamily{U}  {MnSymbolC}{}
\DeclareSymbolFont{MnSyC}         {U}  {MnSymbolC}{m}{n}
\DeclareFontShape{U}{MnSymbolC}{m}{n}{
	<-6>  MnSymbolC5
	<6-7>  MnSymbolC6
	<7-8>  MnSymbolC7
	<8-9>  MnSymbolC8
	<9-10> MnSymbolC9
	<10-12> MnSymbolC10
	<12->   MnSymbolC12}{}
\DeclareMathSymbol{\powerset}{\mathord}{MnSyC}{180}

\usepackage{tikz}
\usetikzlibrary{calc,decorations.pathmorphing}
\usetikzlibrary{calc,intersections,through,hobby}
\pgfdeclarelayer{background}
\pgfdeclarelayer{foreground}
\pgfdeclarelayer{front}
\pgfsetlayers{background,main,foreground,front}

\usepackage{subcaption}
\newcommand{\qedge}[7]{
	
	\ifx\relax#4\relax
	\def\qoffs{0pt}
	\else
	\def\qoffs{#4}
	\fi
	
	\def\qhedge{
		($#1+#3!\qoffs!-90:#2-#3$) --
		($#2+#1!\qoffs!-90:#3-#1$) --
		($#3+#2!\qoffs!-90:#1-#2$) -- cycle}

	\coordinate (12) at ($#1!\qoffs!90:#2$);
	\coordinate (13) at ($#1!\qoffs!-90:#3$);
	\coordinate (23) at ($#2!\qoffs!90:#3$);
	\coordinate (21) at ($#2!\qoffs!-90:#1$);
	\coordinate (31) at ($#3!\qoffs!90:#1$);
	\coordinate (32) at ($#3!\qoffs!-90:#2$);
	
	\def\nqhedge{
		(13) let \p1=($(13)-#1$), \p2=($(12)-#1$) in
		arc[start angle={atan2(\y1,\x1)}, delta angle={atan2(\y2,\x2)-atan2(\y1,\x1)-360*(atan2(\y2,\x2)-atan2(\y1,\x1)>0)}, x radius=\qoffs, y radius=\qoffs] --
		(21) let \p1=($(21)-#2$), \p2=($(23)-#2$) in
		arc[start angle={atan2(\y1,\x1)}, delta angle={atan2(\y2,\x2)-atan2(\y1,\x1)-360*(atan2(\y2,\x2)-atan2(\y1,\x1)>0)}, x radius=\qoffs, y radius=\qoffs] --
		(32) let \p1=($(32)-#3$), \p2=($(31)-#3$) in
		arc[start angle={atan2(\y1,\x1)}, delta angle={atan2(\y2,\x2)-atan2(\y1,\x1)-360*(atan2(\y2,\x2)-atan2(\y1,\x1)>0)}, x radius=\qoffs, y radius=\qoffs] --
		cycle}
	
	\ifx\relax#5\relax
	\def\qlwidth{1pt}
	\else
	\def\qlwidth{#5}
	\fi
	
	\ifx\relax#7\relax
	\fill \nqhedge;
	\else
	\fill[#7]\nqhedge;
	\fi
	
	\ifx\relax#6\relax
	\draw[line width=\qlwidth,rounded corners=\qoffs]\nqhedge;
	\else
	\draw[line width=\qlwidth,#6]\nqhedge;
	\fi
}

\newsavebox\vdegbox
\savebox\vdegbox{\tikz{
		\draw[black,fill=black] (90:1) circle (.35);
		\draw[black,line width=0.10cm] (210:1) circle (.30);
		\draw[black,line width=0.10cm] (330:1) circle (.30);
		\draw[opacity=0] (0:1.2) circle (0.1);
}}

\newsavebox\vvbox
\savebox\vvbox{\tikz{
		\draw[black,line width=0.10cm] (90:1) circle (.30);
		\draw[black,fill=black] (210:1) circle (.35);
		\draw[black,fill=black] (330:1) circle (.35);
		\draw[opacity=0] (0:1.2) circle (0.1);
}}

\newsavebox\pdegbox
\savebox\pdegbox{\tikz{
		\draw[black,line width=0.10cm] (90:1) circle (.30);
		\draw[black,fill=black] (210:1) circle (.35);
		\draw[black,fill=black] (330:1) circle (.35);
		\draw[black,line width=0.28cm ] (210:1) -- (330:1);
		\draw[opacity=0] (0:1.2) circle (0.1);
}}

\newsavebox{\vvvbox}
\savebox{\vvvbox}{\tikz{
		\draw[black,fill=black] (90:1) circle (.35);
		\draw[black,fill=black] (210:1) circle (.35);
		\draw[black,fill=black] (330:1) circle (.35);
		\draw[opacity=0] (0:1.2) circle (0.1);
}}
\newcommand{\vvv}{\mathord{\scaleobj{1.2}{\scalerel*{\usebox{\vvvbox}}{x}}}}
\newcommand{\pivvv}{\pi_{\vvv}}

\newsavebox\evbox
\savebox\evbox{\tikz{
		\draw[black,fill=black] (90:1) circle (.35);
		\draw[black,fill=black] (210:1) circle (.35);
		\draw[black,fill=black] (330:1) circle (.35);
		\draw[black,line width=0.28cm ] (210:1) -- (330:1);
		\draw[opacity=0] (0:1.2) circle (0.1);
}}

\newsavebox\eebox
\savebox\eebox{\tikz{
		\draw[black,fill=black] (90:1) circle (.35);
		\draw[black,fill=black] (210:1) circle (.35);
		\draw[black,fill=black] (330:1) circle (.35);
		\draw[black,line width=0.28cm ] (90:1) -- (330:1);
		\draw[black,line width=0.28cm ] (90:1) -- (210:1);
		\draw[opacity=0] (0:1.2) circle (0.1);
}}

\newsavebox\eeebox
\savebox\eeebox{\tikz{
		\draw[black,fill=black] (90:1) circle (.35);
		\draw[black,fill=black] (210:1) circle (.35);
		\draw[black,fill=black] (330:1) circle (.35);
		\draw[black,line width=0.28cm ] (90:1) -- (330:1);
		\draw[black,line width=0.28cm ] (90:1) -- (210:1);
		\draw[black,line width=0.28cm ] (210:1) -- (330:1);
		\draw[opacity=0] (0:1.2) circle (0.1);
}}

\let\epsilon=\varepsilon

\let\rho=\varrho
\let\theta=\vartheta

\newcommand{\cA}{\mathcal{A}}

\newcommand{\cP}{\mathcal{P}}

\newtheoremstyle{note}  {4pt}  {4pt}  {\sl}  {}  {\bfseries}  {.}  {.5em}          {}
\newtheoremstyle{introthms}  {3pt}  {3pt}  {\itshape}  {}  {\bfseries}  {.}  {.5em}          {\thmnote{#3}}
\newtheoremstyle{remark}  {2pt}  {2pt}  {\rm}  {}  {\bfseries}  {.}  {.3em}          {}

\theoremstyle{plain}
\newtheorem{theorem}{Theorem}[section]
\newtheorem{lemma}[theorem]{Lemma}

\theoremstyle{note}
\newtheorem{dfn}[theorem]{Definition}

\theoremstyle{remark}

\usepackage{accents}

\usepackage{lineno}
\newcommand*\patchAmsMathEnvironmentForLineno[1]{
	\expandafter\let\csname old#1\expandafter\endcsname\csname #1\endcsname
	\expandafter\let\csname old#1\expandafter\endcsname\csname end#1\endcsname
	\renewenvironment{#1}
	{\linenomath\csname old#1\endcsname}
	{\csname oldend#1\endcsname\endlinenomath}}

\AtBeginDocument{
}

\newcommand{\ex}{\text{ex}}

\begin{document}
	
	\title[Beyond the broken tetrahedron]{Beyond the broken tetrahedron}
	
	\author[A.~Y.~Chen]{August Y. Chen}
	\address{Cornell University, Ithaca, New York, United States of America}
	\email{ayc74@cornell.edu}
	
	\author[B.~Sch\"{u}lke]{Bjarne Sch\"{u}lke}
	\address{Extremal Combinatorics and Probability Group, Institute for Basic Science, Daejeon, South Korea}
	\email{schuelke@ibs.re.kr}

	
	\begin{abstract}
		
		Here we consider the hypergraph Tur\'an problem in uniformly dense hypergraphs as was suggested by Erd\H{o}s and S\'os.
		Given a $3$-graph $F$, the uniform Tur\'an density $\pivvv(F)$ of $F$ is defined as the supremum over all $d\in[0,1]$ for which there is an $F$-free uniformly $d$-dense $3$-graph, where uniformly $d$-dense means that every linearly sized subhypergraph has density at least $d$.
		Recently, Glebov, Kr\'al', and Volec and, independently, Reiher, R\"odl, and Schacht proved that $\pivvv(K_4^{(3)-})=\frac{1}{4}$, solving a conjecture by Erd\H{o}s and S\'os.
		Despite substantial attention, the uniform Tur\'an density is still only known for very few hypergraphs.
		In particular, the problem due to Erd\H{o}s and S\'os to determine~$\pivvv(K_4^{(3)})$ remains wide open.
		
		In this work, we determine the uniform Tur\'an density of the $3$-graph on five vertices that is obtained from $K_4^{(3)-}$ by adding an additional vertex whose link forms a matching on the vertices of $K_4^{(3)-}$.
		Further, we point to two natural intermediate problems on the way to determining $\pivvv(K_4^{(3)})$, and solve the first of these.
	\end{abstract}
	\maketitle
	\section{Introduction}
	
	Tur\'an's paper~\cite{T:41} in which he shows that complete, balanced,~$(r-1)$-partite graphs maximise the number of edges among all~$K_r$-free graphs is often seen as the starting point of extremal combinatorics.
	Already in this paper, he asked for similar results for hypergraphs.
	More than $80$ years later, still little is known in this regard.
	
	A~$k$-uniform hypergraph (or~$k$-graph)~$H=(V,E)$ consists of a vertex set~$V$ and an edge set~$E\subseteq V^{(k)}=\{e\subseteq V:\vert e\vert=k\}$.
	The extremal number for~$n\in\mathds{N}$ and a~$k$-uniform hypergraph~$F$ is the maximum number of edges in an~$F$-free~$k$-uniform hypergraph on~$n$ vertices and it is denoted by~$\ex(n,F)$.
	The Tur\'an density of~$F$ is~$\pi(F)=\lim_{n\to\infty}\frac{\ex(n,F)}{\binom{n}{k}}$ (this limit was shown to exist by monotonicity in~\cite{KNS:64}).
	The aforementioned result by Tur\'an was later generalised to the Erd\H{o}s-Stone-Simonovits theorem~\cite{ES:46,ES:66}, which determines the Tur\'an density of any graph~$F$ to be~$\frac{\chi(F)-2}{\chi(F)-1}$, where~$\chi(F)$ is the chromatic number of~$F$.
	
	For hypergraphs, only few (non-trivial) Tur\'an densities are know, one prominent example being the Fano plane~\cite{DF:00,FS:05,KS:05,BR:19}.
	The Tur\'an density of the complete~$3$-uniform hypergraph on four vertices,~$K_4^{(3)}$, is still not known, although it is widely conjectured to be~$5/9$.
	In fact, the problem is even open for the~$3$-uniform hypergraph that consists of three edges on four vertices, denoted by~$K_4^{(3)-}$.
	For more on the hypergraph Tur\'an problem, we refer to the survey by Keevash~\cite{K:11}.
	
	In this work, we consider a variant of the hypergraph Tur\'an problem that was suggested by Erd\H{o}s and S\'os~\cite{E:90,ES:82} who asked for the maximum density of~$F$-free hypergraphs~$H$ if we require~$H$ to be uniformly dense.
	More precisely, given~$d\in[0,1]$ and~$\eta>0$, a~$3$-uniform hypergraph~$H=(V,E)$ is called \emph{uniformly~$(d,\eta)$-dense} if for every~$U\subseteq V$, we have~$\vert U^{(3)}\cap E\vert\geq d\binom{\vert U\vert}{3}-\eta\vert V\vert^3$.
	Then the uniform Tur\'an density of a~$3$-uniform hypergraph~$F$ is defined as
	\begin{align*}
		\pivvv(F)=\sup\{d\in[0,1]:&\text{ for every }\eta>0\text{ and }n\in\mathds{N},\text{ there is an}\\ &F\text{-free, uniformly }(d,\eta)\text{-dense }3\text{-graph }H\text{ with }\vert V(H)\vert\geq n\}\,.
	\end{align*}
	
	Erd\H{o}s and S\'os specifically asked to determine~$\pivvv(K_4^{(3)})$ and~$\pivvv(K_4^{(3)-})$.
	Similarly as with the original Tur\'an density, these problems turned out to be very difficult.
	Only recently, Glebov, Kr\'al', and Volec~\cite{GKV:16} and Reiher, R\"odl, and Schacht~\cite{RRS:18} independently solved the latter, showing that~$\pivvv(K_4^{(3)-})=1/4$ and thus confirming a conjecture by Erd\H{o}s and S\'os.
	We refer to Reiher's survey~\cite{R:20} for a full description of the landscape of extremal problems in uniformly dense hypergraphs.
	Further, Reiher, R\"odl, and Schacht~\cite{RRS:182} \emph{characterised} all~$3$-uniform hypergraphs~$F$ with~$\pivvv(F)=0$ and showed that the minimum positive uniform Tur\'an density of any~$3$-uniform hypergraph is at least~$\frac{1}{27}$.
	Very recently, Garbe, Kr\'al', and Lamaison~\cite{GKL:21} indeed found~$3$-uniform hypergraphs whose uniform Tur\'an density is~$\frac{1}{27}$.
	Also very recently, Buci\'c, Cooper, Kr\'al', Mohr, and Correia~\cite{BCKMC:21} determined the uniform Tur\'an density of all tight~$3$-uniform cycles of length at least five.
	% Note that by supersaturation~$\pivvv(F)=\pivvv(F(t))$ for every~$3$-uniform hypergraph~$F$ and~$t\in\mathds{N}$, where~$F(t)$ is the~$t$-blow-up of~$F$, the hypergraph obtained from~$F$ by replacing every vertex of~$F$ by~$t$ copies of itself.
	% We say that~$F'$ is a blow-up of~$F$ if~$F'\subseteq F(t)$ for some~$t\in\mathds{N}$.
	% Thus, by determining~$\pivvv(F)$ for some~$3$-uniform hypergraph~$F$, also the uniform Tur\'an density of any blow-up of~$F$ is determined.
	Despite significant attention (see~\cite{R:20,BCL:22,BPRRS:22}), there are few hypergraphs~$F$ for which~$\pivvv(F)$ is known. 
	
	In this work, we add another example, namely~$F_{\star}$, which we define as the hypergraph on five vertices with edge set~$\{123, 124, 134, 125, 345\}$, see Figure~\ref{fig:Fstar}.
	In other words,~$F_{\star}$ is the~$3$-uniform hypergraph obtained from a~$K_4^{(3)-}$ by adding an additional vertex whose link is a matching.
	
	\begin{figure}
		\centering
		\begin{tikzpicture}[scale=1.0]
			\coordinate (e) at (-2, 0);
			\coordinate (a) at (0, 1); 
			\coordinate (c) at (0, -1);
			\coordinate (b) at (2, 1);
			\coordinate (d) at (2, -1);
			\begin{pgfonlayer}{front}
				\foreach \i in {a, b, c, d, e}
				\fill  (\i) circle (2pt);
				
				\node at (-2.3,0) {$5$};
				\node at (0, 1.3) {$1$};
				\node at (0, -1.3) {$3$};
				\node at (2.3, 1.3) {$2$};
				\node at (2.3, -1.3) {$4$};
				
			\end{pgfonlayer}
			
			\qedge{(a)}{(b)}{(c)}{4.5pt}{1.5pt}{red!70!black}{red!70!black,opacity=0.2};
			\qedge{(a)}{(b)}{(d)}{4.5pt}{1.5pt}{red!70!black}{red!70!black,opacity=0.2};
			\qedge{(a)}{(d)}{(c)}{4.5pt}{1.5pt}{red!70!black}{red!70!black,opacity=0.2};
			\qedge{(a)}{(b)}{(e)}{4.5pt}{1.5pt}{red!70!black}{red!70!black,opacity=0.2};
			\qedge{(c)}{(e)}{(d)}{4.5pt}{1.5pt}{red!70!black}{red!70!black,opacity=0.2};
			
		\end{tikzpicture}
		\caption{The hypergraph $F_{\star}$.}
		\label{fig:Fstar}
	\end{figure}
	
	% \begin{figure}
		%     \centering
		%     \begin{tikzpicture}[scale=1.5]
			%     \coordinate (a) at (0, -1); 
			%     \coordinate (b) at (2, 0);
			%     \coordinate (c) at (-2, 0);
			%     \coordinate (d) at (0.5, 2);
			% 	\coordinate (e) at (-0.5, 2);
			%     \begin{pgfonlayer}{front}
				%     \foreach \i in {a, b, c, d, e}
				% 		\fill  (\i) circle (2pt);
				
				%     \node at (0,-1.2) {$2$};
				%     \node at (2.2, 0) {$3$};
				%     \node at (-2.2, 0) {$1$};
				%     \node at (0.5, 2.2) {$5$};
				%     \node at (-0.5, 2.2) {$4$};
				
				%     \end{pgfonlayer}
			
			%     \qedge{(b)}{(a)}{(d)}{4.5pt}{1.5pt}{red!70!black}{red!70!black,opacity=0.2};
			%     \qedge{(a)}{(c)}{(d)}{4.5pt}{1.5pt}{red!70!black}{red!70!black,opacity=0.2};
			%     \qedge{(b)}{(c)}{(d)}{4.5pt}{1.5pt}{red!70!black}{red!70!black,opacity=0.2};
			
			%     \qedge{(b)}{(a)}{(e)}{4.5pt}{1.5pt}{red!70!black}{red!70!black,opacity=0.2};
			%     \qedge{(a)}{(c)}{(e)}{4.5pt}{1.5pt}{red!70!black}{red!70!black,opacity=0.2};
			%     \qedge{(b)}{(c)}{(e)}{4.5pt}{1.5pt}{red!70!black}{red!70!black,opacity=0.2};
			
			%     \end{tikzpicture}
		%     \caption{Examples of graphs where it is simple to see $\pivvv(F)=\frac14$.}
		%     \label{fig:my_label}
		% \end{figure}
	
	\begin{theorem}\label{theorem:main}
		We have~$\pivvv(F_{\star})=1/4$.
	\end{theorem}

        We establish this result by developing a novel method to establish the existence of a hypergraph $F$ in a uniformly $(d,\eta)$-dense hypergraph $H$, building off the techniques of Reiher, R\"odl, and Schacht~\cite{RRS:18}.
        We provide a brief sketch of our arguments at the end of this section.
        Since this article was first made available online, our methods have been used by Li, Lin, Wang, and Zhou~\cite{LLWZ:23} to describe a large family of~$3$-uniform hypergraphs all of which have uniform Tur\'{a}n density~$1/4$. 
        Surprisingly, an adaptation of our method has also been used by Piga, Sales, and the second author~\cite{PSS:23} to determine the \emph{codegree} Tur\'{a}n density of any tight cycle minus an edge.
    
	One potential difficulty in moving from determining~$\pivvv(K_4^{(3)-})$ to~$\pivvv(K_4^{(3)})$, which has been described as perhaps the most urgent problem in this area~\cite{R:20}, is that there are no~$3$-uniform hypergraphs~$F$ with~$K_4^{(3)-}\subsetneq F\subsetneq K_4^{(3)}$ that could serve as intermediate steps.
	However, Reiher, R\"odl, and Schacht reduced the problem of determining the uniform Tur\'an density of any hypergraph to a problem in \emph{reduced hypergraphs} (see Theorem~\ref{thm:reduction}).
	The terminology of reduced hypergraphs gives rise to two natural problems `between' determining~$\pivvv(K_4^{(3)-})$ and~$\pivvv(K_4^{(3)})$.
	Our second result solves the first of these.
	Since the notation needed is not the lightest, we postpone the introduction of the notation to Section~\ref{sec:prelim} and ask the reader unfamiliar with it to read it first before continuing here.
	
	Suppose~$\cA=(I,\cP^{ij},\cA^{ijk})$ is a $(d,\vvv)$-dense reduced hypergraph and~$(\lambda,\varphi)$ is a reduced map from~$K_4^{(3)-}$ (with vertex set~$[4]$ and apex at vertex~$1$) to~$\cA$.
	Let~$\lambda(j)=i_j$ and~$\varphi(jk)=\alpha_{jk}$ for~$jk\in[4]^{(2)}$.
	Now note that if~$\alpha_{23}\alpha_{24}\alpha_{34}\in E(\cA^{i_2i_3i_4})$ (that is, if there is an edge in~$\cA^{i_2i_3i_4}$ which intersects~$\{\alpha_{23},\alpha_{24},\alpha_{34}\}$ in three vertices) then~$(\lambda,\varphi)$ would actually be a reduced map from~$K_4^{(3)}$ to~$\cA$.
	If~$\cA$ is~$(d,\vvv)$-dense for any~$d>0$, then~$\cA^{i_2i_3i_4}$ does contain edges (which intersect~$\{\alpha_{23},\alpha_{24},\alpha_{34}\}$ in zero or more vertices).
	This leads to the question which density~$d$ we need to guarantee that there is an edge in~$\cA^{i_2i_3i_4}$ which intersects~$\{\alpha_{23},\alpha_{24},\alpha_{34}\}$ in at least one (or in at least two) vertices.
	In this sense, moving from the problem for~$K_4^{(3)-}$ to the problem for~$K_4^{(3)}$ can be described as `gluing' first one, then two, and eventually three of the vertices in the constituent~$\cA^{i_2i_3i_4}$. 
        As mentioned earlier, upon gluing three vertices, then one obtains the emergence of the desired $K_4^{(3)}$.
	The setting in~$\cA$ with one vertex `glued' is shown in Figure \ref{fig:onevtxglue}.
	
	Our second result states that an intersection in at least one vertex (or the first gluing) is already guaranteed if the density is larger than~$1/4$ (which is clearly best possible). 
    More precisely, we prove the following.
    
	\begin{theorem}\label{theorem:onevtxglue}
		For~$\varepsilon>0$, there is some~$M\in\mathds{N}$ such that in every~$(\frac{1}{4}+\varepsilon,\vvv)$-dense reduced hypergraph~$\cA=([M],\cP^{ij},\cA^{ijk})$, there are indices~$i_1,i_2,i_3,i_4\in [M]$ and vertices~$\alpha_{jk}\in\cP^{i_j i_k}$ for~$jk\in[4]^{(2)}$, as well as vertices~$\alpha'_{23}\in\cP^{i_2 i_3}$ and~$\alpha'_{24}\in\cP^{i_2 i_4}$, such that $\alpha_{13}\alpha_{14}\alpha_{34}\in E(\cA^{i_1i_3i_4})$, $\alpha_{12}\alpha_{13}\alpha_{23} \in E(\cA^{i_1i_2i_3})$, $\alpha_{12}\alpha_{14}\alpha_{24}\in E(\cA^{i_1i_2i_4})$, and $\alpha'_{23}\alpha'_{24}\alpha_{34}\in E(\cA^{i_2i_3i_4})$.
	\end{theorem}
	Here, $\alpha_{34}$ corresponds to the vertex resulting from gluing (see Figure~\ref{fig:onevtxglue}).
	
	To determine $\pivvv(K_4^{(3)})$, an interesting next step is to determine the optimal density~$d$ such that we find such a structure with two vertices glued, i.e., only requiring one of the vertices~$\alpha_{23}'$ and~$\alpha_{24}'$.
	
	The proofs of Theorem~\ref{theorem:main} and Theorem~\ref{theorem:onevtxglue} both utilize our novel method and could be formulated as one.
    For both results, we need to `glue' one vertex in the reduced hypergraph (for an illustration, see Figure~\ref{fig:Finreduced} for Theorem~\ref{theorem:main} and Figure~\ref{fig:onevtxglue} for Theorem~\ref{theorem:onevtxglue} respectively).
	For clarity however, we give the key steps separately in Section~\ref{sec:proofmain} (from Lemma~\ref{lemma:rowprep} onward) and Section~\ref{sec:proofglueing}, respectively.

    \subsection*{Our approach}
    At a high level, our method proceeds by iteratively preparing `rows' in the reduced hypergraph in a mutually compatible fashion, and subsequently applying the pigeonhole principle across many rows.
    This entails the emergence of substructures in the reduced hypergraph which intersect in one vertex, which is precisely the `glued' vertex.
    To prepare a single row we utilise ideas from the work of Reiher, R\"odl, and Schacht~\cite{RRS:18}.
    The crucial new step, however, is to combine the structures from several rows by the pigeonhole principle.
    In order to make this possible, the preparation of one row has to be done more carefully than what a simple application of~\cite{RRS:18} would give.   
    
    It is natural to expect that this method can be used to determine the uniform Tur\'an density of other 3-uniform hypergraphs which arise similarly by `gluing' a vertex in the reduced hypergraph. 
    Indeed, since the first publication of this article, Li, Lin, Wang, and Zhou~\cite{LLWZ:23} used our method to describe a family of~$3$-uniform hypergraphs all of which have uniform Tur\'an density~$\frac{1}{4}$.
    The aforementioned application of our method to the codegree Tur\'an problem by Piga, Sales, and the second author from~\cite{PSS:23} is somewhat surprising, though, since usually different variants of the hypergraph Tur\'an density behave quite differently and require separate approaches.
	
	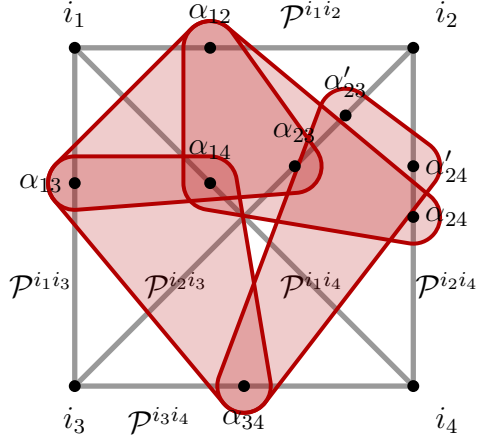
\begin{figure}
		\centering
		\begin{tikzpicture}[scale=2.25]
			\coordinate (1) at (0, 1); 
			\coordinate (3) at (0, -1);
			\coordinate (2) at (2, 1);
			\coordinate (4) at (2, -1);
			
			\coordinate (alpha_{13}) at (0, 0.2);
			\coordinate (alpha_{12}) at (0.8, 1);
			\coordinate (alpha_{14}) at (0.8, 0.2);
			\coordinate (alpha_{23}) at (1.3, 0.3);
			\coordinate (alpha_{24}) at (2, 0);
			\coordinate (alpha_{34}) at (1, -1);
			\coordinate (alpha'_{23}) at (1.6, 0.6);
			\coordinate (alpha'_{24}) at (2, 0.3);
			\begin{pgfonlayer}{front}
				\foreach \i in {1,2,3,4,alpha_{14},alpha_{12},alpha_{13},alpha_{23},alpha_{24},alpha_{34}, alpha'_{23}, alpha'_{24}}
				\fill  (\i) circle (1pt);
				
				\node at (0, 1.2) {$i_1$};
				\node at (0, -1.2) {$i_3$};
				\node at (2.2, 1.2) {$i_2$};
				\node at (2.2, -1.2) {$i_4$};
				
				\node at (-0.2, 0.2) {$\alpha_{13}$};
				\node at (0.8, 1.2) {$\alpha_{12}$};
				\node at (0.8, 0.4) {$\alpha_{14}$};
				\node at (1.3, 0.5) {$\alpha_{23}$};
				\node at (2.2, 0) {$\alpha_{24}$};
				\node at (1, -1.2) {$\alpha_{34}$};
				\node at (1.6, 0.8) {$\alpha'_{23}$};
				\node at (2.2, 0.3) {$\alpha'_{24}$};
				%\node at (
				
			\end{pgfonlayer}
			\draw[black!40!white, line width=2pt] (1) -- (2);
			\draw[black!40!white, line width=2pt] (2) -- (3);
			\draw[black!40!white, line width=2pt] (3) -- (4);
			\draw[black!40!white, line width=2pt] (4) -- (1);
			\draw[black!40!white, line width=2pt] (1) -- (3);
			\draw[black!40!white, line width=2pt] (2) -- (4);
			
			\node at (-0.2, -0.4) {$\cP^{i_1i_3}$};
			\node at (1.4, 1.2) {$\cP^{i_1i_2}$};
			\node at (1.4, -0.4) {$\cP^{i_1i_4}$};
			\node at (0.6, -0.4) {$\cP^{i_2i_3}$};
			\node at (2.2, -0.4) {$\cP^{i_2i_4}$};
			\node at (0.5, -1.2) {$\cP^{i_3i_4}$};
			
			%     \draw[blue!60!blue, line width=2pt] (alpha_{12}) -- (alpha_{13});
			% 	\draw[blue!60!blue, line width=2pt] (alpha_{12}) -- (alpha_{14});
			% 	\draw[blue!60!blue, line width=2pt] (alpha_{13}) -- (alpha_{14});
			
			\qedge{(alpha_{13})}{(alpha_{14})}{(alpha_{34})}{4.5pt}{1.5pt}{red!70!black}{red!70!black,opacity=0.2};
			\qedge{(alpha_{24})}{(alpha_{14})}{(alpha_{12})}{4.5pt}{1.5pt}{red!70!black}{red!70!black,opacity=0.2};
			\qedge{(alpha_{23})}{(alpha_{13})}{(alpha_{12})}{4.5pt}{1.5pt}{red!70!black}{red!70!black,opacity=0.2};
			\qedge{(alpha'_{23})}{(alpha'_{24})}{(alpha_{34})}{4.5pt}{1.5pt}{red!70!black}{red!70!black,opacity=0.2};
		\end{tikzpicture}
		\caption{The reduced image of $K_4^{(3)-}$ with an additional edge and one vertex `glued' ($\alpha_{34})$.}
		\label{fig:onevtxglue}
	\end{figure}
	
	\section{Preliminaries}\label{sec:prelim}
    When dealing with singleton sets, we may omit the parentheses, for instance, given a set~$I$ and~$r\in I$, we may simply write~$I\setminus r$ etc.
    
	When proving~$\pivvv(K_4^{(3)-})=\frac{1}{4}$, Reiher, R\"odl, and Schacht~\cite{RRS:18} used the hypergraph regularity lemma.
	After ``regularising'' a hypergraph, the problem reduces to an extremal problem in the reduced hypergraph.
	We will follow this approach and use the terminology introduced in~\cite{RRS:18} and laid out systematically by Reiher in his survey~\cite{R:20}.

	For a hypergraph~$H$, we denote the number of its edges by~$e(H)=\vert E(H)\vert$ and the number of its vertices by~$v(H)=\vert V(H)\vert$.
	Given a set~$I$, we denote by~$I^{(\ell)}$ the set of all~$\ell$-element subsets of~$I$.
	When considering such subsets, we often omit the parentheses and write for instance~$ijk$ for an element of~$I^{(3)}$, instead of~$\{i,j,k\}$.
	For a hypergraph~$H$, we define the shadow of~$H$ as~$\partial H=\{e\setminus x:x\in e\in E(H)\}$.
	
	\begin{dfn}
		Let~$I$ be a set of indices, for all~$ij\in I^{(2)}$, let~$\cP^{ij}$ be a set of vertices which are disjoint for~$ij\neq i'j'$, and for~$ijk\in I^{(3)}$, let~$\cA^{ijk}$ be a~$3$-partite~$3$-graph with partition classes~$\cP^{ij}$,~$\cP^{ik}$, and~$\cP^{jk}$.
		Then the~$\binom{\vert I\vert}{2}$-partite $3$-graph~$\cA$ with vertex set~$V(\cA)=\bigcup_{ij\in I^{(2)}}\cP^{ij}$ and edge set~$E(\cA)=\bigcup_{ijk\in I^{(3)}}E(\cA^{ijk})$ is called a reduced hypergraph with index set~$I$, vertex classes~$\cP^{ij}$ and constituents~$\cA^{ijk}$.
	\end{dfn}
	For brevity, we often simply write `let~$\cA=(I,\cP^{ij},\cA^{ijk})$ be a reduced hypergraph' instead of `let~$\cA$ be are reduced hypergraph with index set~$I$, vertex classes~$\cP^{ij}$ and constituents~$\cA^{ijk}$'.
	
	Next, let us take a look at how the density condition in our host hypergraphs translates to a density condition in the respective reduced hypergraphs.
	
	\begin{dfn}
		Let~$\cA=(I,\cP^{ij},\cA^{ijk})$ be a reduced hypergraph and let~$d\in [0,1]$ be a real number.
		If~$e(\cA^{ijk})\geq d\vert\cP^{ij}\vert\vert\cP^{ik}\vert\vert\cP^{jk}\vert$ holds for all~$ijk\in I^{(3)}$, we say that~$\cA$ is~$(d,\vvv)$-dense.
	\end{dfn}
	
	The following definition of a reduced map describes the setting in which a hypergraph~$F$ can be embedded into a host hypergraph~$H$ (by using the hypergraph regularity method) in terms of the reduced hypergraph of~$H$.
	
	\begin{dfn}
		A reduced map from a~$3$-graph~$F$ to a reduced hypergraph~$\cA=(I,\cP^{ij},\cA^{ijk})$ is a pair~$(\lambda,\varphi)$ such that
		\begin{enumerate}
			\item $\lambda: V(F)\to I$ and~$\varphi: \partial F\to V(\cA)$;
			\item for~$uv\in\partial F$, we have~$\lambda(u)\neq\lambda(v)$ and~$\varphi(uv)\in \cP^{\lambda(u)\lambda(v)}$;
			\item if~$uvw\in E(F)$, then~$\varphi(uv)\varphi(uw)\varphi(vw)\in E(\cA^{\lambda(u)\lambda(v)\lambda(w)})$.
		\end{enumerate}
		
		If there is a reduced map from~$F$ to~$\cA$, we say that~$\cA$ contains a reduced image of~$F$, otherwise~$\cA$ is called~$F$-free.
	\end{dfn}
	
	Now the following theorem due to Reiher, R\"odl, and Schacht (and appearing in this formulation in Reiher's survey~\cite{R:20}) reduces the problem of determining~$\pivvv(F)$ to a problem about reduced hypergraphs.
	
	\begin{theorem}\label{thm:reduction}
		For a~$3$-graph~$F$, we have
		\begin{align*}
			\pivvv(F)=\sup\{d\in[0,1]:&\text{ For every }M\in\mathds{N}\text{, there is a } (d,\vvv)\text{-dense,} \\ &F\text{-free reduced hypergraph }\cA\text{ with index set }[M]\}\,.
		\end{align*}
	\end{theorem}
	
	\section{Proof of Theorem~\ref{theorem:main}}\label{sec:proofmain}
	\begin{proof}[Proof of Theorem~\ref{theorem:main}]
		Since~$K_4^{(3)-}\subseteq F_{\star}$, we have~$1/4=\pivvv(K_4^{(3)-})\leq\pivvv(F_{\star})$.
		
		Due to Theorem~\ref{thm:reduction}, it is thus enough to show that for every~$\varepsilon>0$, there is some large~$M\in\mathds{N}$ such that every~$(\frac{1}{4}+\varepsilon, \vvv)$-dense reduced hypergraph with index set~$[M]$ contains a reduced image of~$F_{\star}$.
		Let thus~$\varepsilon>0$ be given, without loss of generality~$\varepsilon\ll 1$, and choose~$M,M_*,m\in\mathds{N}$, and~$\delta>0$ such that
		\begin{align}\label{eq:hierarchy}
			M^{-1}\ll M_*^{-1}\ll m^{-1}\ll\delta\ll\varepsilon\ll 1\,.
		\end{align}
		Then let~$\cA$ be a~$(\frac{1}{4}+\varepsilon, \vvv)$-dense reduced hypergraph with index set~$[M]$, vertex classes~$\cP^{ij}$ and constituents~$\cA^{ijk}$.
		We will show that~$\cA$ contains a reduced image of~$F_{\star}$.
		
		% Now, the idea is to repeat the argument from \cite{reiher2018turan} in two repeated stages, and then finish by Pigeonhole. Work with the hierarchy $\epsilon^{-1} \ll \delta^{-1}$, and from here, choose $\widehat{M} \ge 2\max(\lceil\epsilon^{-2}\rceil+1, g^{(-(\lceil\epsilon^{-2}\rceil+1))}(2))$, where $g(x) = \delta(x-2)$ and $g^{(-n)}(x)$ denotes the inverse of $g$ composed $n \ge 1$ times. Finally, take $M=1+(\delta^{-1})^{\widehat{M}+1}$ and choose $m_1 \gg m_* \gg M$. This is fine because Corollary 3.3 in \cite{reiher2018turan} where Hypergraph Regularity is applied allows for $m_1$ to be arbitrary, and since our choice of $\widehat{M}, M$ here just depends on $\epsilon, \delta$ -- which are initially given -- and the inverse of given, fixed functions, specifically $g$. We perform the two iteration processes and then use Pigeonhole in a way that will guarantee the desired structure. 
		
		We begin with a `cleaning process' (until Lemma~\ref{lemma:quadmanytriangles}) similar to the one in~\cite{RRS:18}.
		However, we will allow for a bit more `robustness'.
		Consider any~$ijk\in [M]^{(3)}$ with~$i<j<k$.
		We define the bipartite graph~$Q^i_{jk}$ on the vertex set~$\cP^{ij}\dcup\cP^{ik}$ by placing an edge between~$w \in \mathcal{P}^{ij}$ and~$v \in \mathcal{P}^{ik}$ if there are at least~$\epsilon^2 |\mathcal{P}^{jk}|$ vertices~$u$ in~$\mathcal{P}^{jk}$ such that~$wvu \in E(\mathcal{A}^{ijk})$.
		Similarly, we define the bipartite graph~$Q^k_{ij}$ on the vertex set~$\cP^{ik}\dcup\cP^{jk}$ by placing an edge between~$v \in \mathcal{P}^{ik}$ and~$w \in \mathcal{P}^{jk}$ if there are at least~$\epsilon^2 |\mathcal{P}^{ij}|$ vertices~$u$ in~$\mathcal{P}^{ij}$ such that~$uvw \in E(\mathcal{A}^{ijk})$.
		We write~$N_{Q^i_{jk}}(x)$ and~$d_{Q^i_{jk}}(x)$ for the neighbourhood and the degree of a vertex~$x$ in the graph~$Q^i_{jk}$, respectively, and similarly we define~$N_{Q^k_{ij}}(x)$ and~$d_{Q^k_{ij}}(x)$. 
		
		\begin{lemma}\label{lemma:sumofsquares}
			For any~$ijk\in[M]^{(3)}$ with~$i<j<k$, we have 
			\[ \sum_{v \in \mathcal{P}^{ik}} d_{Q^i_{jk}}(v) d_{Q^k_{ij}}(v) \ge \big(\frac14+\frac{\epsilon}2\big) |\mathcal{P}^{ij}||\mathcal{P}^{jk}||\mathcal{P}^{ik}| \,.\]
		\end{lemma}
		\begin{proof}
			We use a double-counting argument.
			For~$v \in \cP^{ik}$, set~$T(v)=\{w_1vw_2\in E(\cA^{ijk})\}$ and~$S(v)=\{w_1vw_2\in E(\mathcal{A}^{ijk}):vw_1\in E(Q^i_{jk}),vw_2\in E(Q^k_{ij})\}$,
			\begin{align*}
				&S_1(v)=\{w_1vw_2\in E(\mathcal{A}^{ijk}):w_1\in\cP^{ij},vw_1\notin E(Q^i_{jk})\}\text{, and}\\
				&S_2(v)=\{w_1vw_2\in E(\mathcal{A}^{ijk}):w_2\in\cP^{jk},vw_2\notin E(Q^k_{ij})\}\,.
			\end{align*}
			Note that by double-counting, we have the following.
			\begin{align*}
				\big(\frac{1}{4}+\varepsilon\big)\vert\cP^{ij}\vert\vert\cP^{jk}\vert\vert\cP^{ik}\vert&\leq\sum_{v\in\cP^{ik}}\vert T(v)\vert\leq\sum_{v\in\cP^{ik}}\vert S(v)\vert+\vert S_1(v)\vert+\vert S_2(v)\vert\\
				&\leq\sum_{v\in\cP^{ik}} d_{Q^i_{jk}}(v) d_{Q^k_{ij}}(v)+2\varepsilon^2\vert\cP^{ij}\vert\vert\cP^{jk}\vert\vert\cP^{ik}\vert
				\,.
			\end{align*}
			Since~$\varepsilon\ll 1$, the assertion follows.
			%-d_{Q^i_{jk}}(v)-d_{Q^k_{ij}}(v)&\leq\sum_{v\in\cP^{ik}} d_{Q^i_{jk}}(v) d_{Q^k_{ij}}(v)+\varepsilon^2\vert\cP^{ij}\vert\vert\cP^{jk}\vert+\varepsilon^2\vert\cP^{ij}\vert\vert\cP^{jk}\vert\\
		\end{proof}
		
		Given~$ijk\in[M]^{(3)}$ with~$i<j<k$, by Lemma \ref{lemma:sumofsquares} and the Cauchy-Schwartz inequality, we have either
		\begin{align}\label{eq:blue}
			\sum_{v \in \mathcal{P}^{ik}} d_{Q^i_{jk}}(v)^2 \ge \big(\frac14+\frac\epsilon2\big) |\mathcal{P}^{ij}|^2 |\mathcal{P}^{ik}|\,,
		\end{align}
		in which case we colour $ijk$ blue, and if this is not true, we must have
		\[ \sum_{v \in \mathcal{P}^{ik}} d_{Q^k_{ij}}(v)^2 \ge \big(\frac14+\frac\epsilon2\big) |\mathcal{P}^{jk}|^2 |\mathcal{P}^{ik}|\,,\]
		in which case we colour $ijk$ red. 
		
		By~(\ref{eq:hierarchy}) and by Ramsey's theorem, there is a subset of~$M_*$ indices, without loss of generality~$[M_*]$, such that all~$3$-subsets of this set are coloured with the same colour, say blue.
		
		There is one more step in the `cleaning process', for which we apply Ramsey's theorem to a different colouring of the~$3$-subsets of~$[M_*]$.
		For~$ijk\in[M_*]^{(3)}$ with~$i<j<k$ and $r \in \mathbb{N}$, let 
		\[ S_{jk}^i(r) = \{x \in \mathcal{P}^{ik}: d_{Q^i_{jk}}(x) \ge \big(\frac12+r \delta\big) |\mathcal{P}^{ij}|\}\,.\]
		Since~$ijk$ was coloured in blue,~\eqref{eq:blue} entails
		\begin{align*}
			\big(\frac14+\frac{\epsilon}2\big) |\mathcal{P}^{ij}|^2 |\mathcal{P}^{ik}|\leq\sum_{v \in \mathcal{P}^{ik}} d_{Q^i_{jk}}(v)^2\leq \big(\frac12+\delta\big)^2 |\mathcal{P}^{ij}|^2|\mathcal{P}^{ik}\setminus S_{jk}^i(1)|+ |\mathcal{P}^{ij}|^2|S_{jk}^i(1)|\,.
		\end{align*}
		
		Because of~\eqref{eq:hierarchy}, we infer that~$\vert S_{jk}^i(1)\vert\geq\delta\vert\cP^{ik}\vert$.
		
		Clearly the $S_{jk}^i(r)$'s are decreasing and $S_{jk}^i(r) = \emptyset$ for $r > \frac{1}{2\delta}$, so we may colour $ijk$ by the maximum~$r\in[\frac{1}{2\delta}]$ (w.l.o.g., we assume that~$\frac{1}{2\delta}$ is an integer) for which $|S_{jk}^i(r)| \ge \delta |\mathcal{P}^{ik}|$.
		
		Again by Ramsey's theorem, as $m \ll M_{*}$ and after potentially reordering the indices, we can suppose that $[m]$ is a blue clique in the first coloring and an $r_{*}$-coloured clique in the second coloring, for some~$r_*\in[\frac{1}{2\delta}]$.
		The following is the conclusion of our cleaning.
		\begin{align*}(\star)
			\text{ Every }ijk\in[m]^{(3)}\text{ with }i<j<k\text{ satisfies~\eqref{eq:blue} and }\vert S_{jk}^i(r_*)\vert\geq\delta\vert\cP^{ik}\vert>\vert S_{jk}^i(r_*+1)\vert\,.
		\end{align*}
		
		Now, for all $i \in [m]$, define the graph~$Q^i = \bigcup_{1\leq i<j < k\leq m} Q^i_{jk}$.
		We need the following auxiliary result, which is a robust version of a result in~\cite{RRS:18}.
		\begin{lemma}\label{lemma:quadmanytriangles}
			For~$i,j_1,j_2,k \in [m]$ with~$i<j_1<j_2<k$, and any~$x \in S_{j_1 k}^i(r_*) \cap S_{j_2 k}^i(r_*)$, there are at least~$\delta |\mathcal{P}^{ij_1}||\mathcal{P}^{ij_2}|$ triangles~$xyz$ in~$Q^i$ with~$y \in \mathcal{P}^{ij_2}, z \in \mathcal{P}^{ij_1}$.
		\end{lemma}
		\begin{proof}
			Let $A_{j_1}$ be the neighbors of $x$ in $\mathcal{P}^{ij_1}$, and let $A_{j_2}$ be the neighbors of $x$ in $\mathcal{P}^{ij_2}$.
			By definition of $x$ (and~$S^i_{jk}$), we have that~$|A_{j_1}| \ge (\frac12+r_* \delta) |\mathcal{P}^{ij_1}|, |A_{j_2}| \ge (\frac12+r_* \delta) |\mathcal{P}^{ij_2}|$.
			
			Also, denote $B_{j_1}=\mathcal{P}^{ij_1}\setminus A_{j_1}$ and $B_{j_2}=\mathcal{P}^{ij_2}\setminus A_{j_2}$. Now let $C = B_{j_2}\cap S_{j_1j_2}^i(r_*+1)$.
			By~$(\star)$, we know $|C| < \delta |\mathcal{P}^{ij_2}|$, and by definition of~$S_{j_1j_2}^i(r_*+1)$, we have $d_{Q^i_{j_1 j_2}}(y)  < (\frac12+(r_*+1)\delta)|\mathcal{P}^{ij_1}|$ for all $y \in B_{j_2}\setminus C$.
			
			Note that any edge between~$A_{j_1}$ and~$A_{j_2}$ gives rise to a triangle of the kind the statement asks for, so suppose for the sake of contradiction that there are less than~$\delta |\mathcal{P}^{ij_1}||\mathcal{P}^{ij_2}|$ edges between~$A_{j_1}$ and~$A_{j_2}$.

			Then, when setting~$d'_{Q^i_{j_1 j_2}}(y)=\vert N_{Q^i_{j_1 j_2}}(y)\cap A_{j_1}\vert$ for all~$y \in A_{j_2}$, we have~$\sum_{y \in A_{j_2}} d'_{Q^i_{j_1 j_2}}(y) < \delta |\mathcal{P}^{ij_1}||\mathcal{P}^{ij_2}|$.
			Since each~$d'_{Q^i_{j_1 j_2}}(y) \le |\mathcal{P}^{ij_1}|$, we also know~$\sum_{y\in A_{j_2}} \big(d'_{Q^i_{j_1 j_2}}(y)\big)^2 \le \delta |\mathcal{P}^{ij_1}|^2|\mathcal{P}^{ij_2}|$.
			By our conditions on~$|A_{j_1}|$, we have~$d_{Q^i_{j_1 j_2}}(y) - d'_{Q^i_{j_1 j_2}}(y) \le |B_{j_1}| \le (\frac12-r_* \delta)|\mathcal{P}^{ij_1}|$. Hence
			\begin{align*}
				\sum_{y \in A_{j_2}} d_{Q^i_{j_1 j_2}}(y)^2 &= \sum_{y \in A_{j_2}} \big(d_{Q^i_{j_1 j_2}}(y)-d'_{Q^i_{j_1 j_2}}(y)+d'_{Q^i_{j_1 j_2}}(y)\big)^2 \\
				% &\le |A_{j_2}| \big(\frac12-r_* \delta\big)^2 |\mathcal{P}^{ij_1}|^2 + \delta |\mathcal{P}^{ij_1}|^2|\mathcal{P}^{ij_2}| + 2\sum_{y \in A_{j_2}} d'_{Q^i_{j_1 j_2}}(y) |B_{j_1}| \\
				&\le |A_{j_2}| (\frac12-r_* \delta)^2 |\mathcal{P}^{ij_1}|^2 + \delta |\mathcal{P}^{ij_1}|^2|\mathcal{P}^{ij_2}| + 2 (\frac12-r_*\delta) \delta |\mathcal{P}^{ij_1}|^2  |\mathcal{P}^{ij_2}| \\
				&< |A_{j_2}| (\frac12-r_* \delta)^2 |\mathcal{P}^{ij_1}|^2 + 2\delta |\mathcal{P}^{ij_1}|^2|\mathcal{P}^{ij_2}|\,.
			\end{align*}
			
			Now recalling~$(\star)$ and splitting the sum yields
			\begin{align*}
				&\big(\frac14+\frac{\epsilon}2\big) |\mathcal{P}^{ij_1}|^2 |\mathcal{P}^{ij_2}|\leq \sum_{x \in \cP^{ij_2}} d_{Q^i_{j_1 j_2}}(x)^2\\
				\leq&\sum_{x \in B_{j_2}\setminus C} d_{Q^i_{j_1 j_2}}(x)^2+\sum_{x \in A_{j_2}} d_{Q^i_{j_1 j_2}}(x)^2+\sum_{x \in C} d_{Q^i_{j_1 j_2}}(x)^2\\
				\le &|B_{j_2}| \big(\frac12+(r_*+1)\delta\big)^2 |\mathcal{P}^{ij_1}|^2 + |A_{j_2}| \big(\frac12-r_* \delta\big)^2 |\mathcal{P}^{ij_1}|^2 + 3\delta |\mathcal{P}^{ij_1}|^2|\mathcal{P}^{ij_2}|\,.
			\end{align*}
			Since~$\vert A_{j_2}\vert+\vert B_{j_2}\vert=\vert \cP^{ij_2}\vert$ and~$\vert B_{j_2}\vert\leq(\frac{1}{2}-r_{\star}\delta)\vert\cP^{ij_2}\vert$, the right-hand side is maximised when~$\vert B_{j_2}\vert=(\frac{1}{2}-r_{\star}\delta)\vert\cP^{ij_2}\vert$ and~$\vert A_{j_2}\vert=(\frac{1}{2}+r_{\star}\delta)\vert\cP^{ij_2}\vert$.
			Cancelling~$\vert\mathcal{P}^{ij_1}\vert^2 \vert\mathcal{P}^{ij_2}\vert$ entails
			\begin{align*}
				\frac14+\frac{\epsilon}2 &\le \big(\frac12-r_* \delta\big) \big(\frac12+(r_*+1)\delta\big)^2 + \big(\frac12+r_* \delta\big) \big(\frac12-r_* \delta\big)^2 + 3\delta  \\
				&\le \big(\frac12-r_* \delta\big) \big(\frac12+r_* \delta\big)^2 + \big(\frac12+r_* \delta\big) \big(\frac12-r_* \delta\big)^2 + 6\delta  \\
				&< \frac14 + 6 \delta \,.
			\end{align*}
			Since $\delta \ll \epsilon$, this yields a contradiction.
		\end{proof}
		
		In our proof, we will iteratively prepare a number of `rows', each of which is a subgraph of a~$Q^r$, by finding many triangles that are `compatible' with each other.
		By definition of~$Q^r$, for each edge in~$e\in E(Q^r)$, there exists a small subset~$X$ in some vertex class of~$Q^{r'}$ for some~$r'>r$, such that every~$x\in X$ together with~$e$ forms an edge in~$\cA$.
        This subset~$X$ can be thought of as the `projection' of the row~$r$ onto the row~$r'$.
		After preparing many rows, we will show that some of these projections have to intersect.
		From this, the reduced image of~$F_{\star}$ will arise (see Figure~\ref{fig:Fperrows}).
		
		The following lemma encapsulates the preparation of one row which is `paid for' by shrinking the index set.
		\begin{lemma}\label{lemma:rowprep}
			Let~$I\subseteq [m]$ with~$\vert I\vert> \frac{2}{\delta^2}$,~$m\in I$, and~$r=\min I$.
			Then there are~$J\subseteq I\setminus r$, with~$m\in J$ and~$\vert J\vert\geq\delta^2 \vert I\vert$,~$x\in \cP^{rm}$,~$y\in \cP^{rr'}$, where~$r'=\min J$, and~$z_j\in \cP^{rj}$ for all~$j\in J\setminus \{r',m\}$ such~$xz_jy$ forms a triangle in~$Q^r$.
		\end{lemma}
		\begin{proof}
			First, we find~$x\in \cP^{rm}$.
			As~$|S_{jm}^r(r_{*})| \ge \delta |\mathcal{P}^{rm}|$ for all $j \in I\setminus\{m, r\}$, it follows that there is some set $I' \subseteq I\setminus\{r, m\}$ such that $|I'| \ge \delta(|I|-2)$ and some $x \in \bigcap_{j \in I'} S_{jm}^r(r_*)$.
			
			Let $r' = \min I'$; next, we find $y \in \cP^{rr'}$. For any $j \in I'$, let $A_j$ be the set of neighbours of $x$ in $\cP^{rj}$.
			Since $x \in S_{jm}^r(r_*) \cap S_{r'm}^r(r_*)$, by Lemma \ref{lemma:quadmanytriangles}, there are at least $\delta |\cP^{rj}||\cP^{rr'}|$ edges in $Q^r$ between $A_j$ and $A_{r'}$.
			For~$j\in I'\setminus\{r'\}$, let $D_j$ be the set of vertices in $A_{r'}$ that are connected by some edge in~$Q^r$ to a vertex in $A_j$.
			Then,~$|D_j| \ge \frac{\delta |\cP^{rj}||\cP^{rr'}|}{|\cP^{rj}|}=\delta |\cP^{rr'}|$ and double-counting yields that there is some vertex $y \in A_{r'}$ that lies in $D_j$ for at least $\frac{\delta |\cP^{rr'}|(|I'|-1)}{|\cP^{rr'}|}=\delta(|I'|-1)$ such $j \in I'$.
			Denote the set of such~$j \in I'$ by~$I''$ and let~$J=I''\cup\{m,r'\}$.
			
			Now~$r'=\min J$ and for all~$j \in I''=J\setminus\{r',m\}$, we have that~$y \in D_j$.
			Hence, there is some vertex~$z_j \in A_j$ such that~$z_j y$ is an edge in~$Q^r$.
			Since~$z_j\in A_j$ and~$y\in A_{r'}$, we know that~$xz_j$ and~$xy$ are also edges in~$Q^r$. Thus, for all~$j \in J\setminus \{r',m\}$,~$xz_jy$ is a triangle in~$Q^r$.
			The lower bound on~$|J|$ follows as~$|J| \ge \delta (|I'|-1)+2\geq\delta^2(|I|-2)-\delta+2 \ge \delta^2 |I|$.
		\end{proof}
		
		Now, we apply Lemma~\ref{lemma:rowprep} iteratively, in each step preparing a row and shrinking the index set.
		First, let~$I_0=[m]$, and~$r_1=\min I_0$.
		Lemma~\ref{lemma:rowprep} provides some~$I_1\subseteq I_0\setminus r_1$ with~$m\in I_1$ and~$\vert I_1\vert\geq\delta^2 \vert I_0\vert$,~$x_1\in\cP^{r_1m}$,~$y_1\in\cP^{r_1r_2}$, where~$r_2=\min I_1$, and~$z_{1j}\in\cP^{r_1j}$ for all~$j\in I_1\setminus\{r_2,m\}$ such that~$x_1z_{1j}y_1$ forms a triangle in~$Q^{r_1}$.
		
		Next, we apply Lemma~\ref{lemma:rowprep} again, but this time to~$I_1$.
		Generally, assume that for some~$i\in[\frac{2}{\varepsilon^2}]$, we have defined sets~$I_0,\dots,I_i\subseteq[m]$, indices~$r_1,\dots r_{i+1}\in[m]$, and for all~$j\in[i]$ and~$k\in I_j\setminus\{r_{j+1},m\}$, vertices~$x_j\in\cP^{r_{j}m}$,~$y_j\in\cP^{r_{j}r_{j+1}}$, and~$z_{jk}\in\cP^{r_{j}k}$ such that the following holds.
		For all~$j\in[i]$, we have~$I_{j}\subseteq I_{j-1}\setminus r_{j}$,~$m\in I_{j}$,~$r_{j+1}=\min I_j$, and~$\vert I_{j}\vert\geq\delta^2\vert I_{j-1}\vert$, and further for all~$k\in I_j\setminus\{r_{j+1},m\}$,~$x_jz_{jk}y_j$ forms a triangle in~$Q^{r_{j}}$.
		Then applying Lemma~\ref{lemma:rowprep} to~$I_i$ yields~$I_{i+1}\subseteq I_{i}\setminus r_{i+1}$ with~$m\in I_{i+1}$ and~$\vert I_{i+1}\vert\geq\delta^2\vert I_i\vert$,~$x_{i+1}\in\cP^{r_{i+1}m}$,~$y_{i+1}\in\cP^{r_{i+1}r_{i+2}}$, where~$r_{i+2}=\min I_{i+1}$, and~$z_{{i+1}j}\in\cP^{r_{i+1}j}$ for all~$j\in I_{i+1}\setminus\{r_{i+2},m\}$ such that~$x_{i+1}z_{{i+1}j}y_{i+1}$ forms a triangle in~$Q^{r_{i+1}}$.
		
		Note that we can perform this process~$\frac{2}{\varepsilon^2}+1$ times since in the last step (and thus in all steps prior) we have that~$\vert I_{\frac{2}{\varepsilon^2}}\vert\geq (\delta^2)^{\frac{2}{\varepsilon^2}} m> \frac{2}{\delta ^2}$ (using the hierarchy~\eqref{eq:hierarchy}) and can thus we can still construct the set~$I_{\frac{2}{\varepsilon^2}+1}$ with~$\vert I_{\frac{2}{\varepsilon^2}+1}\vert\geq 3$.
		
		After this iteration procedure, let the maximum element of~$I_{\frac{2}{\varepsilon^2}+1}\setminus m$ be $m'$.
		For all $i \in[\frac{2}{\varepsilon^2}+1]$, since $I_{\frac{2}{\varepsilon^2}+1} \subseteq I_i$, there is some triangle $x_i z_{im'} y_i$ in $Q^{r_{i}}$ and we set~$e_i=x_iz_{im'}$.
		Note that~$e_i$ is between $\cP^{r_i m}$ and $\cP^{r_i m'}$.
		Since~$e_i\in E(Q^{r_i})$, we know that there is a set~$H(e_i)\subseteq \cP^{m'm}$ with~$\vert H(e_i)\vert\geq\varepsilon^2\vert\cP^{m'm}\vert$ such that~$x_iz_{im'}v\in E(\cA^{r_im'm})$ for every~$v\in H(e_i)$.
		Thus, there are at least three indices~$i,j,k\in[\frac{2}{\varepsilon^2}+1]$ such that~$H(e_i)\cap H(e_j)\cap H(e_k)\neq\emptyset$. Without loss of generality, assume $i<j<k$.
		
		Let~$v\in H(e_i)\cap H(e_j)\cap H(e_k)$ and note that~$r_k>r_j \ge r_{i+1}$ and that since $\vert I_{\frac{2}{\varepsilon^2}+1} \vert \ge 3$ and $m = \max I_{\frac{2}{\varepsilon^2}+1}$, $m' = \max I_{\frac{2}{\varepsilon^2}+1} \setminus m$, we know $r_k < m'$.
		Thus, there are~$z_{im'}\in\cP^{r_im'}$ and~$z_{i r_k}\in\cP^{r_i r_k}$ such that~$x_iz_{im'}y_i$ is a triangle in~$Q^{r_i}$ and~$z_{i r_k}y_i\in E(Q^{r_i})$ (we do not need the other edges in the respective triangle).
		Now these edges imply the existence of vertices $u_1 \in \mathcal{P}^{r_{i+1} m}, u_2 \in \mathcal{P}^{r_{i+1} m'}, u_3 \in \mathcal{P}^{r_{i+1} r_k}$ such that $x_{i}y_{i}u_1 \in E(\mathcal{A}^{r_ir_{i+1}m})$,~$z_{im'}y_{i}u_2\in E(\mathcal{A}^{r_ir_{i+1}m'})$, and~$z_{i r_k}y_{i}u_3 \in E(\mathcal{A}^{r_ir_{i+1}r_k})$.
		Together with the edges~$x_iz_{im'}v\in E(\cA^{r_im'm})$ and~$x_kz_{km'}v\in E(\cA^{r_km'm})$, this yields a reduced embedding of~$F_{\star}$. This iteration process and the resulting embedding is illustrated in Figure \ref{fig:Fperrows}.
		
		Indeed, as shown in Figure \ref{fig:Finreduced}, the edges mentioned show that when setting $\lambda(1)=r_i$, $\lambda(2)=r_{i+1}$, $\lambda(3)=m$, $\lambda(4)=m'$, and~$\lambda(5)=r_k$ and~$\varphi(12)=y_i$, $\varphi(13)=x_i$, $\varphi(14)=z_{im'}$, $\varphi(15)=z_{ir_k}$, $\varphi(23)=u_1$, $\varphi(24)=u_2$, $\varphi(25)=u_3$, $\varphi(34)=v$, $\varphi(35)=x_k$, and~$\varphi(45)=z_{km'}$, then~$(\lambda,\varphi)$ is a reduced embedding of~$F_{\star}$ into~$\cA$.
		
		\begin{figure}
			\centering
			\begin{tikzpicture}[xscale=0.9,yscale=0.75]
				
				\coordinate (x_i) at (0, 10); 
				\coordinate (z_im') at (2, 9.5);
				\coordinate (z_irk) at (4, 10.5); 
				\coordinate (y_i) at (6, 10);
				
				\coordinate (u_1) at (0, 7);
				\coordinate (u_2) at (2, 7);
				\coordinate (u_3) at (4, 7);
				
				\coordinate (x_k) at (0, 4);
				\coordinate (z_km') at (2, 4);
				
				\coordinate (v) at (0, 1);
				
				\begin{pgfonlayer}{front}
					
					\foreach \i in {x_i, z_im', z_irk, y_i, u_1, u_2, u_3, x_k, z_km', v}
					\fill  (\i) circle (2pt);
					
					\node at (-0.5, 10) {$x_i$};
					\node at (1.5, 9.5) {$z_{im'}$};
					\node at (3.5, 10.5) {$z_{ir_k}$};
					\node at (6.5, 10) {$y_i$};
					
					\node at (-0.5, 7) {$u_1$};
					\node at (1.5, 7) {$u_2$};
					\node at (3.5, 7) {$u_3$};
					
					\node at (-0.5, 4) {$x_k$};
					\node at (2.5, 4) {$z_{km'}$};
					
					\node at (-0.5, 1) {$v$};
					
					\node at (0, 11.5) {$\mathcal{P}^{r_i m}$};
					\node at (2, 11.5) {$\mathcal{P}^{r_i m'}$};
					\node at (4, 11.5) {$\mathcal{P}^{r_i r_k}$};
					\node at (6, 11.5) {$\mathcal{P}^{r_i r_{i+1}}$};
					
					\node at (0, 8.5) {$\mathcal{P}^{r_{i+1} m}$};
					\node at (2, 8.5) {$\mathcal{P}^{r_{i+1} m'}$};
					\node at (4, 8.5) {$\mathcal{P}^{r_{i+1} r_k}$};
					
					\node at (0, 5.5) {$\mathcal{P}^{r_k m}$};
					\node at (2, 5.5) {$\mathcal{P}^{r_k m'}$};
					
					\node at (0, 2.5) {$\mathcal{P}^{m m'}$};
					
					\node at (-2, 10) {$\text{Row }r_i$};
					\node at (-2, 7) {$\text{Row }r_{i+1}$};
					\node at (-2, 4) {$\text{Row }r_k$};
					\node at (-2, 1) {$\text{Row }m'$};
					
				\end{pgfonlayer}
				\begin{pgfonlayer}{background}
					\draw[black!40!white, line width=2pt] (0,9) -- (0,11);
					\draw[black!40!white, line width=2pt] (2,9) -- (2,11);
					\draw[black!40!white, line width=2pt] (4,9) -- (4,11);
					\draw[black!40!white, line width=2pt] (6,9) -- (6,11);
					
					\draw[black!40!white, line width=2pt] (0,6) -- (0,8);
					\draw[black!40!white, line width=2pt] (2,6) -- (2,8);
					\draw[black!40!white, line width=2pt] (4,6) -- (4,8);
					
					\draw[black!40!white, line width=2pt] (0,3) -- (0,5);
					\draw[black!40!white, line width=2pt] (2,3) -- (2,5);
					
					\draw[black!40!white, line width=2pt] (0,0) -- (0,2);
					
					\draw[blue!60!blue, line width=2pt] (x_i) -- (z_im');
					\draw[blue!60!blue, line width=2pt] (x_i) -- (y_i);
					\draw[blue!60!blue, line width=2pt] (z_im') -- (y_i);
					\draw[blue!60!blue, line width=2pt] (z_irk) -- (y_i);
					\draw[blue!60!blue, line width=2pt] (x_k) -- (z_km');
					
				\end{pgfonlayer}
				\qedge{(y_i)}{(u_3)}{(z_irk)}{4.5pt}{1.5pt}{red!70!black}{red!70!black,opacity=0.2};
				\qedge{(y_i)}{(u_2)}{(z_im')}{4.5pt}{1.5pt}{red!70!black}{red!70!black,opacity=0.2};
				\qedge{(y_i)}{(u_1)}{(x_i)}{4.5pt}{1.5pt}{red!70!black}{red!70!black,opacity=0.2};
				\qedge{(z_km')}{(v)}{(x_k)}{4.5pt}{1.5pt}{red!70!black}{red!70!black,opacity=0.2};
				\qedge{(z_im')}{(v)}{(x_i)}{4.5pt}{1.5pt}{red!70!black}{red!70!black,opacity=0.2};
				
			\end{tikzpicture}
			\caption{$F_{\star}$ embedded across several prepared rows in the reduced hypergraph~$\cA$. Blue lines represent edges in~$Q^r$.}
			\label{fig:Fperrows}
		\end{figure}
		
		\begin{figure}
			\centering
			\begin{tikzpicture}[scale=1.5]
				\coordinate (rk) at (-2, 0);
				\coordinate (ri) at (0, 1); 
				\coordinate (m) at (0, -1);
				\coordinate (ri') at (2, 1);
				\coordinate (m') at (2, -1);
				
				\coordinate (xi) at (0, 0.2);
				\coordinate (yi) at (0.8, 1);
				\coordinate (z_im') at (0.8, 0.2);
				\coordinate (z_irk) at (-1, 0.5);
				\coordinate (u1) at (1.3, 0.3);
				\coordinate (u2) at (2, 0);
				\coordinate (u3) at (-0.7, 1.6);
				\coordinate (v) at (1, -1);
				\coordinate (xk) at (-1, -0.5);
				\coordinate (z_km') at (-0.7, -1.6);
				\begin{pgfonlayer}{front}
					\foreach \i in {rk, ri, m, ri', m', xi, yi, z_im', z_irk, u1, u2, u3, v, xk, z_km'}
					\fill  (\i) circle (1pt);
					
					\node at (-2.2,0) {$r_k$};
					\node at (0, 1.2) {$r_i$};
					\node at (0, -1.2) {$m$};
					\node at (2.2, 1.2) {$r_{i+1}$};
					\node at (2.2, -1.2) {$m'$};
					
					\node at (-0.2, 0.2) {$x_i$};
					\node at (0.8, 1.2) {$y_i$};
					\node at (0.8, 0.4) {$z_{im'}$};
					\node at (-1, 0.7) {$z_{i r_k}$};
					\node at (1.3, 0.5) {$u_1$};
					\node at (2.2, 0) {$u_2$};
					\node at (-0.7, 1.8) {$u_3$};
					\node at (1, -1.2) {$v$};
					\node at (-1, -0.7) {$x_k$};
					\node at (-0.7, -1.8) {$z_{km'}$};
				\end{pgfonlayer}
				
				\draw[black!40!white, line width=2pt] (ri) -- (ri');
				\draw[black!40!white, line width=2pt] (ri) -- (m);
				\draw[black!40!white, line width=2pt] (ri) -- (m');
				\draw[black!40!white, line width=2pt] (rk) -- (ri);
				\draw[black!40!white, line width=2pt] (rk) -- (m);
				\draw[black!40!white, line width=2pt] (ri') -- (m');
				\draw[black!40!white, line width=2pt] (m') -- (m);
				\draw[black!40!white, line width=2pt] (m) -- (ri');
				
				\draw[black!40!white, line width=2pt] (rk) to [curve through={(u3)}] (ri');
				\draw[black!40!white, line width=2pt] (rk) to [curve through={(z_km')}] (m');
				
				\draw[blue!60!blue, line width=2pt] (xi) -- (z_im');
				\draw[blue!60!blue, line width=2pt] (xi) -- (yi);
				\draw[blue!60!blue, line width=2pt] (z_im') -- (yi);
				\draw[blue!60!blue, line width=2pt] (z_irk) -- (yi);
				\draw[blue!60!blue, line width=2pt] (xk) -- (z_km');
				
				\qedge{(yi)}{(z_irk)}{(u3)}{4.5pt}{1.5pt}{red!70!black}{red!70!black,opacity=0.2};
				\qedge{(yi)}{(u2)}{(z_im')}{4.5pt}{1.5pt}{red!70!black}{red!70!black,opacity=0.2};
				\qedge{(yi)}{(u1)}{(xi)}{4.5pt}{1.5pt}{red!70!black}{red!70!black,opacity=0.2};
				\qedge{(z_km')}{(xk)}{(v)}{4.5pt}{1.5pt}{red!70!black}{red!70!black,opacity=0.2};
				\qedge{(z_im')}{(v)}{(xi)}{4.5pt}{1.5pt}{red!70!black}{red!70!black,opacity=0.2};
				
			\end{tikzpicture}
			\caption{The reduced image of~$F_{\star}$ in the reduced hypergraph~$\cA$. As before, the gray lines between two indices~$i$ and~$j$ represent the vertex class~$\cP^{ij}$.}
			\label{fig:Finreduced}
		\end{figure}
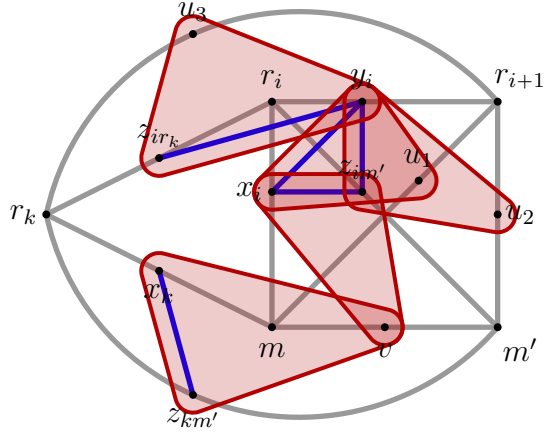
	\end{proof}
	
	\section{Proof of Theorem~\ref{theorem:onevtxglue}}\label{sec:proofglueing}
	
	\begin{proof}
		Let~$\varepsilon>0$ be given, without loss of generality,~$\varepsilon\ll 1$, and choose~$M,M_*,m\in\mathds{N}$ and~$\delta>0$ such that
		\begin{align}\label{eq:hierarchy2}
			M^{-1}\ll M_*^{-1}\ll m^{-1}\ll\delta\ll\varepsilon\ll 1\,.
		\end{align}
		Then let~$\cA$ be a~$(\frac{1}{4}+\varepsilon)$-dense reduced hypergraph with index set~$[M]$, vertex classes~$\cP^{ij}$ and constituents~$\cA^{ijk}$.
		We need to show that there are four indices~$i_1,i_2,i_3,i_4\in [M]$ and vertices~$x_{\alpha\beta}\in\cP^{i_{\alpha}i_{\beta}}$ for~$\alpha\beta\in[4]^{(2)}$ as well as~$x_{23}'\in\cP^{i_2i_3}$ and~$x_{24}'\in\cP^{i_2i_4}$ such that~$x_{12}x_{13}x_{23}\in E(\cA^{i_1i_2i_3})$, $x_{12}x_{14}x_{24}\in E(\cA^{i_1i_2i_4})$, $x_{13}x_{14}x_{34}\in E(\cA^{i_1i_3i_4})$, and~$x_{23}'x_{24}'x_{34}\in E(\cA^{i_2i_3i_4})$.
		
		We proceed as in the last section until Lemma~\ref{lemma:rowprep}.
		Instead of this lemma, we need the following variant.
		
		\begin{lemma}\label{lem:rowprepglue}
			Let~$m_1,m_2\in\mathds{N}$ with~$m_1\geq \frac{m_2+1}{\delta^{m_2}}$ and let~$I\subseteq [m]$ with~$\vert I\vert\geq m_1$,~$m\in I$ and~$r=\min I$.
			Then there are~$J\subseteq I\setminus r$ with~$m\in J$ and~$\vert J\vert\geq m_2$,~$x\in\cP^{rm}$ and~$z_{k}\in\cP^{rk}$ for all~$k\in J\setminus\{r',m\}$, where~$r'=\min J$, such that for all~$jk\in (J\setminus m)^{(2)}$ with~$j<k$, there is some~$y\in\cP^{rj}$ such that~$xz_{k}y$ forms a triangle in~$Q^r$.
		\end{lemma}
		
		\begin{proof}
			We will choose the~$z_{k}$ iteratively and with each choice the set of considered indices will shrink.
			As~$|S_{jm}^r(r_{*})| \ge \delta |\mathcal{P}^{rm}|$ for all $j \in I\setminus\{m, r\}$, it follows that there is some set $I' \subseteq I\setminus\{r, m\}$ such that $|I'| \ge \delta(|I|-2)$ and some $x \in \bigcap_{j \in I'} S_{jm}^r(r_*)$.
			For any $k \in I'$, let $A_k$ be the set of neighbours of $x$ in $\cP^{rk}$.
			Since for all~$jk\in I'^{(2)}$ with~$j<k$, we have~$x \in S_{km}^r(r_*) \cap S_{jm}^r(r_*)$, by Lemma \ref{lemma:quadmanytriangles}, there are at least $\delta |\cP^{rk}||\cP^{rj}|$ edges in $Q^r$ between $A_k$ and $A_j$.
			Let now~$k_1=\max I'$ and let~$D_j$ be the set of vertices in~$\cP^{rk_1}$ which lie in a triangle between~$A_{k_1}$ and~$A_j$.
			Note that~$\vert D_j\vert\geq \delta\vert\cP^{rk_1}\vert$.
			Therefore, there is some set~$I_1'\subseteq I'\setminus k_1$ with~$\vert I_1'\vert\geq \delta(\vert I'\vert-1)$ and some~$z_{k_1}\in\bigcap_{j\in I_1'}D_j$.
			Set~$I_1=I_1'\cup\{k_1,m\}$ and note that~$\vert I_1\vert\geq \delta^2\vert I\vert$ and that for all~$j\in I_1$ with~$j<k_1$, there is some~$y\in\cP^{rj}$ such that~$xz_{k_1}y$ forms a triangle in~$Q^r$.
			
			Now suppose that we have found indices~$m>k_1>\dots>k_i$ in~$I$ for some~$i<m_2$, a set~$I_i\subseteq I\setminus r$ with~$\{m,k_1,\dots,k_i\}\subseteq I_i$ and~$\vert I_i\vert\geq \delta^{i+1}\vert I\vert$, and vertices~$z_k\in\cP^{rk}$ for all~$k\in\{k_1,\dots,k_i\}$ such that for all~$k\in\{k_1,\dots,k_i\}$ and~$j\in I_i$ with~$j<k$, there is some~$y\in\cP^{rj}$ such that~$xz_ky$ forms a triangle in~$Q^r$.
			Then let~$I_{i+1}'=I_i\setminus\{m,k_1,\dots,k_i\}$ and note that~$\vert I_{i+1}'\vert\geq\delta^{i+1} m_1-(i+1)\geq1$.
			Further, let~$k_{i+1}=\max I_{i+1}'$ and for~$j\in I_{i+1}'\setminus k_{i+1}$, let~$D_j'$ be the set of vertices in~$\cP^{rk_{i+1}}$ which lie in a triangle between~$A_{k+1}$ and~$A_j$.
			As before, we have that~$\vert D_j'\vert\geq \delta\vert\cP^{rk_{i+1}}\vert$ and thus, there is some set~$I_{i+1}''\subseteq I_{i+1}'\setminus k_{i+1}$ and some~$z_{k_{i+1}}\in\bigcap_{j\in I_{i+1}''}D_j'$ (if~$I_{i+1}''=\emptyset$, pick~$z_{k_{i+1}}\in\cP^{rk_{i+1}}$ arbitrarily).
			Set~$I_{i+1}=I_{i+1}''\cup\{m,k_1,\dots k_{i+1}\}$ and note that $$\vert I_{i+1}\vert\geq \delta(\vert I_{i}\vert-(i+2))+i+2\geq \delta \vert I_i\vert\geq\delta^{i+2}\vert I\vert\,,$$ and that for all~$k\in\{k_1,\dots k_{i+1}\}$ and~$j\in I_{i+1}$ with~$j<k$, there is some~$y\in\cP^{rj}$ such that~$xz_ky$ forms a triangle in~$Q^r$.
			
			Eventually, we set~$J=\{m,k_1,\dots,k_{m_2}\}$ and note that~$J$,~$x$, and~$z_k$ for~$k\in J$ have the asserted properties.
		\end{proof}
		
		Similarly as in the last section, we will now conclude the proof by an iteration process, this time repeatedly applying Lemma~\ref{lem:rowprepglue}.
		Let~$m_1,\dots m_{\frac{1}{\varepsilon^2}+1}\in \mathds{N}$ with
		\begin{align}\label{eq:hierarchym}
			m^{-1}=m_1^{-1}\ll m_2^{-1}\ll\dots\ll m_{\frac{1}{\varepsilon^2}+1}\ll\delta\,,
		\end{align}
		and note that this is compatible with the hierarchy~\eqref{eq:hierarchy2}.
		Set~$I_0=[m]$,~$r_1=1$ (note that we will use~$I_i$ here in a different sense than in the proof of Lemma~\ref{lem:rowprepglue} but essentially in the same sense as in the proof of Theorem~\ref{theorem:main}).
		Generally, assume that we have defined, for some~$i\in[\frac{1}{\varepsilon^2}]$, sets~$I_0,\dots,I_i\subseteq[m]$, indices~$r_1,\dots,r_{i+1}\in[m]$, and for each~$i'\in[i]$ and~$k\in I_{i'}\setminus \{r_{i'+1},m\}$, vertices~$x_{i'}\in\cP^{r_{i'}m}$ and~$z_{i'k}\in\cP^{r_{i'}k}$ such that the following holds.
		For all~$i'\in[i]$, we have~$I_{i'}\subseteq I_{i'-1}\setminus r_{i'}$,~$m\in I_{i'}$,~$r_{i'+1}=\min I_{i'}$, and~$\vert I_{i'}\vert\geq m_{i'}$, and further for all~$jk\in (I_{i'}\setminus m)^{(2)}$ with~$j<k$, there is some~$y\in \cP^{r_{i'}j}$ such that~$x_{i'}z_{i'k}y$ forms a triangle in~$Q^{r_{i'}}$.
		
		Then we apply Lemma~\ref{lem:rowprepglue} with~$I_i$ instead of~$I$, and~$m_i$ and~$m_{i+1}$ in place of~$m_1$ and~$m_2$, respectively.
		This yields a set~$I_{i+1}\subseteq I_i\setminus r_{i+1}$ with~$m\in I_{i+1}$ and~$\vert I_{i+1}\vert\geq m_{i+1}$, a vertex~$x_{i+1}\in\cP^{r_{i+1}m}$, and for all~$k\in I_{i+1}\setminus\{r_{i+2},m\}$, where~$r_{i+2}=\min I_{i+1}$, vertices~$z_{i+1k}\in\cP^{r_{i+1}k}$ such that for all~$jk\in (I_{i+1}\setminus m)^{(2)}$ with~$j<k$, there is some~$y\in \cP^{r_{i+1}j}$ such that~$x_{i+1}z_{i+1k}y$ forms a triangle in~$Q^{r_{i+1}}$.
		
		Due to~\eqref{eq:hierarchym}, we can proceed in this manner until~$i=\frac{1}{\varepsilon^2}+1$, that is, we can construct a set~$I_{\frac{1}{\varepsilon^2}+1}$ with~$\vert I_{\frac{1}{\varepsilon^2}+1}\vert\geq m_{\frac{1}{\varepsilon^2}+1}$ in this manner.
		Let~$m'=\max I_{\frac{1}{\varepsilon^2}+1}\setminus m$.
		During the iteration, we found for each~$i\in[\frac{1}{\varepsilon^2}+1]$ vertices~$x_i$ and~$z_{im'}$ such that, in particular,~$x_iz_{im'}\in E(Q^{r_i})$.
		We set~$e_i=x_iz_{im'}$.
		By definition of~$Q^{r_i}$, there is a set~$H(e_i)\subseteq\cP^{m'm}$ with~$\vert H(e_i)\vert\geq\varepsilon^2\vert\cP^{m'm}\vert$ such that~$x_iz_{im'}v\in E(\cA^{r_im'm})$ for every~$v\in H(e_i)$.
		Thus there are two indices~$i,j\in[\frac{1}{\varepsilon^2}+1]$ such that~$H(e_i)\cap H(e_j)\neq\emptyset$.
		W.l.o.g., assume~$i<j$.
		
		Let~$v\in H(e_i)\cap H(e_j)$ and let~$y\in\cP^{r_ir_j}$ such that~$x_iz_{im'}y$ forms a triangle in~$Q_{r_i}$ (such a~$y$ exists due to the choice of~$x_i$ and~$z_{im'}$).
		The definition of~$Q^{r_i}$ implies that there are vertices~$u_1\in\cP^{r_jm}$ and~$u_2\in\cP^{r_jm'}$ such that~$x_iyu_1\in E(\cA^{r_ir_jm})$ and~$z_{im'}yu_2\in E(\cA^{r_ir_jm'})$.
		Since we also have the edges~$x_iz_{im'}v\in E(\cA^{r_im'm})$ and~$x_jz_{jm'}v\in E(\cA^{r_jm'm})$, the indices~$r_i$,~$r_j$,~$m'$, and~$m$ and the vertices~$x_i$,~$z_{im'}$,~$y$,~$x_j$,~$u_1$,~$z_{jm'}$,~$u_2$, and~$v$ form the desired configuration. This is illustrated in Figure \ref{fig:OneVtxGlueRows}.
	\end{proof}
	
	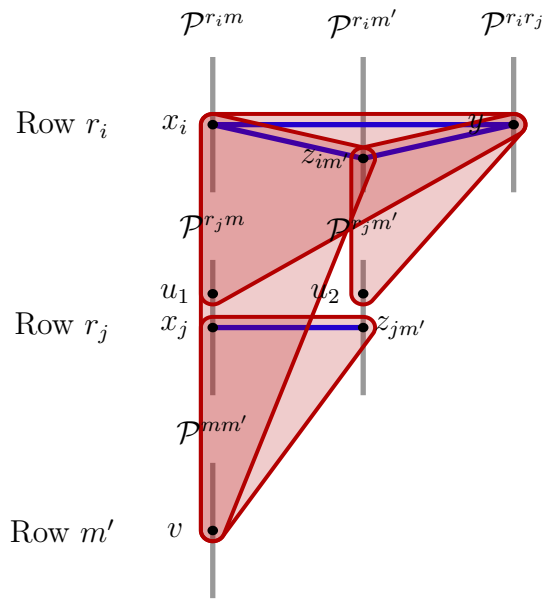
\begin{figure}
		\centering
		\begin{tikzpicture}[xscale=1,yscale=0.9]
			
			\coordinate (x_i) at (0, 10); 
			\coordinate (z_im') at (2, 9.5);
			\coordinate (y) at (4, 10);
			
			\coordinate (u_1) at (0, 7.5);
			\coordinate (u_2) at (2, 7.5);
			\coordinate (x_j) at (0, 7);
			\coordinate (z_jm') at (2, 7);
			
			\coordinate (v) at (0, 4);
			
			\begin{pgfonlayer}{front}
				
				\foreach \i in {x_i, z_im', y, u_1, u_2, x_j, z_jm', v}
				\fill  (\i) circle (2pt);
				
				\node at (-0.5, 10) {$x_i$};
				\node at (1.5, 9.5) {$z_{im'}$};
				\node at (3.5, 10) {$y$};
				
				\node at (-0.5, 7.5) {$u_1$};
				\node at (1.5, 7.5) {$u_2$};
				\node at (-0.5, 7) {$x_j$};
				\node at (2.5, 7) {$z_{jm'}$};
				
				\node at (-0.5, 4) {$v$};
				
				\node at (0, 11.5) {$\mathcal{P}^{r_i m}$};
				\node at (2, 11.5) {$\mathcal{P}^{r_i m'}$};
				\node at (4, 11.5) {$\mathcal{P}^{r_i r_j}$};
				
				\node at (0, 8.5) {$\mathcal{P}^{r_{j} m}$};
				\node at (2, 8.5) {$\mathcal{P}^{r_{j} m'}$};
				
				\node at (0, 5.5) {$\mathcal{P}^{m m'}$};
				
				\node at (-2, 10) {$\text{Row }r_i$};
				\node at (-2, 7) {$\text{Row }r_{j}$};
				\node at (-2, 4) {$\text{Row }m'$};
				
			\end{pgfonlayer}
			\begin{pgfonlayer}{background}
				\draw[black!40!white, line width=2pt] (0,9) -- (0,11);
				\draw[black!40!white, line width=2pt] (2,9) -- (2,11);
				\draw[black!40!white, line width=2pt] (4,9) -- (4,11);
				
				\draw[black!40!white, line width=2pt] (0,6) -- (0,8);
				\draw[black!40!white, line width=2pt] (2,6) -- (2,8);
				
				\draw[black!40!white, line width=2pt] (0,3) -- (0,5);
				
				\draw[blue!60!blue, line width=2pt] (x_i) -- (z_im');
				\draw[blue!60!blue, line width=2pt] (x_i) -- (y);
				\draw[blue!60!blue, line width=2pt] (z_im') -- (y);
				\draw[blue!60!blue, line width=2pt] (x_j) -- (z_jm');
				
			\end{pgfonlayer}
			\qedge{(y)}{(u_2)}{(z_im')}{4.5pt}{1.5pt}{red!70!black}{red!70!black,opacity=0.2};
			\qedge{(y)}{(u_1)}{(x_i)}{4.5pt}{1.5pt}{red!70!black}{red!70!black,opacity=0.2};
			\qedge{(z_im')}{(v)}{(x_i)}{4.5pt}{1.5pt}{red!70!black}{red!70!black,opacity=0.2};
			\qedge{(z_jm')}{(v)}{(x_j)}{4.5pt}{1.5pt}{red!70!black}{red!70!black,opacity=0.2};
			
		\end{tikzpicture}
		\caption{$K_4^{(3)-}$ with one vertex glued embedded in the reduced hypergraph~$\cA$ that we obtain, following the procedure above.}
		\label{fig:OneVtxGlueRows}
	\end{figure}
	
	\subsection*{Acknowledgements}
	The second author thanks Christian Reiher and Mathias Schacht for fruitful discussions and for teaching him~\cite{RRS:18}.
    The second author was partially supported by the Young Scientist Fellowship IBS-R029-Y7.

\end{document}